\documentclass{article}
\usepackage{amsmath,amssymb,amsthm,fullpage,ytableau,graphicx,xcolor,subcaption}

\newtheorem{thm}{Theorem}[section]
\newtheorem{lem}[thm]{Lemma}

\newtheorem{cor}[thm]{Corollary}
\theoremstyle{definition}
\newtheorem{dfn}{Definition}[section]
\newtheorem{exa}{Example}[section]

\newcommand{\NDL}{\operatorname{NDL}}

\title{Neighborhood degree lists of graphs}
\author{Michael D. Barrus\\
\small Department of Mathematics\\
\small University of Rhode Island\\
\small Kingston, RI 02881\\
\small \tt barrus@uri.edu\\
\\
Elizabeth Donovan\\
\small Department of Mathematics and Statistics\\
\small Murray State University\\
\small Murray, KY 42071\\
\small \tt edonovan@murraystate.edu
}

\begin{document}
\maketitle
\begin{abstract}
The neighborhood degree list (NDL) is a graph invariant that refines information given by the degree sequence and joint degree matrix of a graph and is useful in distinguishing graphs having the same degree sequence. We show that the space of realizations of an NDL is connected via a switching operation. We then determine the NDLs that have a unique realization by a labeled graph; the characterization ties these NDLs and their realizations to the threshold graphs and difference graphs.
\end{abstract}

\section{Introduction}
\label{sec: intro}

Though the degree sequence of a graph is one of the simplest possible invariants of a graph, it has attracted considerable interest and yielded beautiful results. Several different tests are known for determining if a list of integers is a degree sequence, and many authors have written about the properties that the graphs having a given degree sequence (the \emph{realizations} of the sequence) can or must have.

In particular, several authors have asked or answered questions concerning the uniqueness of realizations. In the case of strict uniqueness, where there is only one possible realization of a degree sequence once degrees are prescribed for labeled vertices, the degree sequences involved are the threshold sequences; their realizations are called threshold graphs (see the monograph~\cite{MahadevPeled95} for a survey). A more relaxed question of uniqueness requires that there only be one realization of the degree sequence up to isomorphism (the degree sequence $(1,1,1,1)$, for instance, has three distinct realizations but only one up to isomorphism). Degree sequences with realizations from a unique isomorphism class are called unigraphic, and their realizations are unigraphs. For a discussion of unigraphs and a good bibliography, see~\cite{Tyshkevich00}. (When questions of uniqueness are addressed in later sections of this paper, our understanding of uniqueness will be in the former sense, where isomorphism classes are ignored and graphs with distinct edge sets are considered to be distinct.)

The degree sequence is not the only descriptive parameter based on the degrees of vertices.  In~\cite{PatrinosHakimi76}, Patrinos and Hakimi considered \emph{integer-pair sequences}, collections of unordered pairs of integers produced by recording the degrees of the two endpoints of each edge in a graph (or pseudograph, multigraph, etc.). They determined which sequences of integer pairs can be realized by a graph, pseudograph, or multigraph, and  later Das~\cite{Das81}  characterized the integer-pair sequences that correspond to a single graph (up to isomorphism). More recently, authors have studied a reformulation of integer-pair sequences known as the \emph{joint degree matrix}, where multiplicities of integer pairs are recorded in matrix form~\cite{AGM08,CDEM15,StantonPinar12}.

The integer-pair sequence and joint degree matrix yield more information about a graph than a degree sequence does, and in this paper we introduce a new degree-related parameter, the \emph{neighborhood degree list} (NDL), that yields still more. The neighborhood degree list of a graph $G$ is a list
\[\tau(G)=((\tau^1_1,\dots,\tau^1_{d_1}),\dots,(\tau^n_1,\dots,\tau^n_{d_n}))\] whose elements are lists of the degrees in $G$ of the neighbors of a given vertex. For example, if $G$ is the graph obtained by attaching a pendant vertex to a chordless 4-cycle, then $\tau(G) = ((2,2,1),(3,2),(3,2),(2,2),(3))$. Notice how the degree sequence $(3,2,2,2,1)$ of $G$ is apparent from the lengths of the elements of $\tau(G)$ (we call these elements the \emph{component lists}). The order of the main list, together with the order of integers within component lists, is usually of little consequence, though for convenience we will order the integers within a component list from largest to smallest and will list the component lists in descending order of length.

We can represent an NDL graphically by placing the integers it contains into the Young diagram of its degree sequence in much the same way Young tableaux are represented, with the terms of one component list per row. For example, we depict the NDL $((2,2,1),(3,2),(3,2),(2,2),(3))$ from above by the diagram in Figure~\ref{fig: firsttableau}. Note that our orderings of numbers within rows and columns of the diagram does not follow typical monotonicity rules for Young tableaux.
\begin{figure}
\centering
\ytableausetup{centertableaux}
\begin{ytableau}
2 & 2 & 1 \\
3 & 2 \\
3 & 2 \\
2 & 2 \\
3
\end{ytableau}
\caption{A graphical depiction of $((2,2,1),(3,2),(3,2),(2,2),(3))$.}
\label{fig: firsttableau}
\end{figure}

Besides providing more information than the degree sequence and the integer-pair sequence or joint degree matrix, our motivation for studying neighborhood degree lists comes from a few contexts. In situations where it becomes necessary to distinguish between nonisomorphic graphs having the same degree sequence, it may be possible to do so by consulting their NDLs. The first author used NDLs in this way in a proof in~\cite{Barrus12} (see Theorem 2.3 therein). 

We find another application in graph reconstruction, which we now describe. The well known Graph Reconstruction Conjecture, attributed to Kelly~\cite{Kelly42,Kelly57} and Ulam~\cite{Ulam60}, states that every $n$-vertex graph (where $n \geq 3$) is uniquely determined up to isomorphism by the multiset of its induced subgraphs of order $n-1$. These subgraphs are called the \emph{cards} of the graph, and the collection of cards is called the \emph{deck}. Given the deck of an unknown graph $G$, a standard counting argument yields the degree sequence of $G$: first, sum the numbers of edges in each of the cards; denote the result by $s$. The graph $G$ then contains exactly $s/(n-2)$ edges, since each edge appears in all but two of the cards. Given any card of $G$, we can then subtract the number of edges in the card from the number $s/(n-2)$ to determine the degree of the missing vertex. Doing this for each card in turn yields the degree sequence of $G$.

However, the degree sequence is not all we can determine in this way. By comparing the degree sequence of the card and the degree of the missing vertex to the degree sequence of the graph $G$, it is possible to determine the degrees in $G$ of the vertices to which the missing vertex is adjacent. Thus our counting argument yields not only the degree sequence but also the neighborhood degree list of $G$. (It also shows that the Reconstruction Conjecture is true for graphs that are uniquely determined, up to isomorphism, by their NDLs.)

In this paper we initiate a study of the NDL of a graph. In Section~\ref{sec: graphic NDLS} we characterize the NDLs of simple graphs and describe how to construct the realizations of one. In Section~\ref{sec: N-switches} we present an edge-switching operation for transforming one realization of an NDL into any other realization of the same NDL. Finally, in Section~\ref{sec: NDL-unique}, we determine the NDL analogues of threshold sequences and graphs by determining which NDLs have unique realizations by labeled graphs, and which graphs these are.

Throughout the paper, we use $V(G)$ to denote the vertex set of a graph $G$. We denote the degree sequence of $G$ by $\deg(G)$, and the degree in $G$ of a vertex $v$ by $\deg_G(v)$ or by $\deg(v)$, if $G$ is understood from the context.

\section{Feasible tableaux}
\label{sec: graphic NDLS}
In this section we characterize those lists of lists of integers that are neighborhood degree lists of simple graphs. We say that a \emph{tableau} is a list %
\begin{equation}\label{eq:TableauNotation}
T=((\tau^1_1,\dots,\tau^1_{d_1}),\dots,(\tau^n_1,\dots,\tau^n_{d_n}))
\end{equation} %
of $n$ lists of nonnegative integers, where the lengths $d_1,\dots,d_n$ of the component lists are in descending order, as are the terms $\tau^i_1,\dots,\tau^i_{d_i}$ in each component list. 
A tableau is \emph{feasible} if 
each integer appearing in the tableau is equal to one of the terms $d_i$.

We now define some notation that will be used throughout the paper. Let the tableau $T$ in \eqref{eq:TableauNotation} be a feasible tableau, and let $d=(d_1,\dots,d_n)$. For any integer $k$ appearing in $d$, let $V_k$ denote the subset of $\{1,\dots,n\}$ containing elements $i$ for which $d_i = k$.

For each $i \in V_k$ and integer $\ell$ appearing in $d$, let $\mu^{\ell}_{i}$ denote the number of times the term $\ell$ appears in the tableau component list $(\tau_1^i,\dots,\tau_{k}^i)$.

For each integer $k$ appearing in $d$, let $D^k$ be the list consisting of the terms $\mu^k_i$ for all $i \in V_k$; while order of terms in $D^k$ will largely be unimportant, we may stipulate that the terms $\mu^k_i$ be arranged in increasing order of $i$. For distinct integers $k$ and $\ell$ both appearing in $d$, with $k>\ell$, let $D^{k,\ell}$ be the bipartitioned list in which the first part contains the terms $\mu^{\ell}_i$ for all $i \in V_k$, and the second part contains the terms $\mu^k_i$ for all $i \in V_\ell$; terms in both parts may be assumed to be arranged in increasing order of $i$.

As an example, in the tableau $((2,2,1),(3,2),(3,2),(2,2),(3))$ illustrated in the previous section, we may write $D^1=(0)$ and $D^2 = (1,1,2)$ and $D^3=(0)$; also $D^{2,1}=(0,0,0;0)$ and $D^{3,1}=(1;1)$ and $D^{3,2}=(2;1,1,0)$. As shown in Figure~\ref{fig: subtableau}, in a graphical representation of the tableau, computing $D^k$ corresponds to counting those boxes containing a $k$ in rows of length $k$ while $D^{k,\ell}$ consists of first counting boxes containing an $\ell$ in rows of length $k$ followed by counting boxes containing a $k$ in rows of length $\ell$.
\begin{figure}
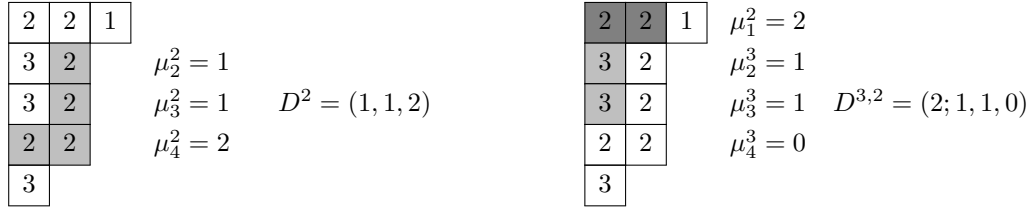

\centering
\ytableausetup{centertableaux}
\begin{ytableau}
2 & 2 & 1 \\
3 & *(lightgray) 2  & \none & \none & \none[\mu_2^2 = 1]\\
3 & *(lightgray) 2 & \none & \none & \none[\mu_3^2 = 1] & \none & \none & \none & \none[D^2=(1,1,2)]\\
*(lightgray) 2 & *(lightgray) 2 & \none & \none & \none[\mu_4^2 = 2]\\
3
\end{ytableau}
\hspace{1in}
\begin{ytableau}
*(gray) 2 & *(gray) 2 & 1 & \none & \none[\mu_1^2=2]\\
*(lightgray)3 &  2  & \none & \none & \none[\mu_2^3 = 1]\\
*(lightgray)3 &  2 & \none & \none & \none[\mu_3^3 = 1] & \none & \none & \none & \none[D^{3,2}=(2;1,1,0)]\\
2 & 2 & \none & \none & \none[\mu_4^3 = 0]\\
3
\end{ytableau}
\caption{Computing $D^k$ and $D^{k,\ell}$ from a tableau.}
\label{fig: subtableau}
\end{figure}

We can now give our characterization of neighborhood degree lists. A bipartitioned graph is a bipartite graph with a fixed partition of its vertex set into partite sets; we indicate the partition in the degree sequence of a bipartitioned graph by listing the degrees of all vertices in one partite set before beginning the other partite set, separating the two parts' degrees with a semicolon.

\begin{thm}
\label{thm: feasible tableau}
Let $T=((\tau^1_1,\dots,\tau^1_{d_1}),\dots,(\tau^n_1,\dots,\tau^n_{d_n}))$ be a feasible
tableau. The tableau $T$ is the NDL of a simple graph if and only if both the following hold:
\begin{enumerate}
\item[(a)] for each distinct term $k$ in $d$, the list $D^k$ is the degree sequence of a simple graph;
\item[(b)] for each pair $k,\ell$ of distinct values appearing in $d$, the list $D^{k,\ell}$ is the degree sequence of a bipartitioned graph.
\end{enumerate}
\end{thm}
\begin{proof}
Suppose first that both conditions (a) and (b) hold for $T$. Let $V$ be the vertex set $\{1,\dots,n\}$, and for each integer $k$ appearing in $d$, let $G_k$ be graph with vertex set $V$ in which all vertices in $V - V_k$ have degree 0 and the induced subgraph with vertex set $V_k$ is a realization of $D^k$ in which each vertex $i \in V_k$ has degree $\mu^k_i$.

Likewise, for each pair $k,\ell$ of distinct integers appearing in $d$, with $k > \ell$, let $G_{k,\ell}$ be a graph with vertex set $V$ in which all vertices in $V-V_k-V_\ell$ have degree 0 and the induced subgraph with vertex set $V_k \cup V_\ell$ is a realization of $D^{k,\ell}$ by a bipartite graph with partite sets $V_k$ and $V_\ell$, where each vertex $i \in V_k$ has degree $\mu^\ell_i$ and each vertex $j \in V_\ell$ has degree $\mu^k_j$.

We may construct a graph $G$ whose NDL is $T$ by taking letting $G$ be the graph with vertex set $\{1,\dots,n\}$ and edge set formed from the union of the edge sets of all graphs $G_k$ and $G_{k,\ell}$ for allowed values of $k,\ell$ from $d$. It is a simple matter to verify that each vertex has a neighborhood degree list that matches the corresponding component list in the tableau $T$.

Conversely, suppose $T$ is the neighborhood degree list of a simple graph $G$ with degree sequence $d=(d_1,\dots,d_n)$. With all notation as above, we may assume that the vertex set of $G$ is $V=\{1,\dots,n\}$ and that for all $i \in V$ vertex $i$ has degree $d_i$ in $G$. Note that for each integer $k$ appearing in $d$ the set $V_k$ then consists of the vertices of $G$ having degree $k$, and each induced subgraph $G[V_k]$ has degree sequence $D^k$. Furthermore, for any distinct integers $k,\ell$ appearing in $d$, the edges of $G$ joining vertices of degree $k$ to vertices of degree $\ell$ form the edge set of a bipartite graph on $V$ having degree sequence $D^{k,\ell}$ in which the sets $V_k$ and $V_\ell$ are independent sets. Thus $D^k$ and $D^{k,\ell}$ are graphic as claimed for all allowed $k,\ell$.
\end{proof}

A number of criteria are known for testing conditions (a) and (b) in Theorem~\ref{thm: feasible tableau}. For part (a), Section 3.1 of the book~\cite{MahadevPeled95} gives nine equivalent criteria for testing whether a list of nonnegative integers is a degree sequence. We mention here two such criteria. The first is the well known result due to Erd\H{o}s and Gallai~\cite{ErdosGallai60} (with a simplification due to Hammer, Ibaraki, and Simeone~\cite{HammerIbarakiSimeone78,HammerIbarakiSimeone81}).

\begin{thm}[\cite{ErdosGallai60,HammerIbarakiSimeone78,HammerIbarakiSimeone81}]
If $d=(d_1,\dots,d_n)$ is a list of nonnegative integers, listed in nonincreasing order, with an even sum, and $m(d) = \max\{i:d_i \geq i-1\}$, then $d$ is the degree sequence of a simple graph if and only if \[\sum_{i = 1}^k d_i \leq k(k-1) + \sum_{i=k+1}^n \min\{k,d_i\}\] for all $k \in \{1,\dots,m(d)\}$.
\end{thm}

Our second test of (a) was observed by Merris~\cite{Merris03} and brings the graphical nature of Young diagrams into play. Given the Young diagram of a partition $\pi$ of a positive integer, let $A(\pi)$ be the shape comprised of the boxes in the diagram whose column index is at least as large as the row index; that is, $A(\pi)$ consists of all boxes lying on or to the right of the main diagonal. Let $B(\pi)$ be the shape comprised of the remaining boxes in the Young diagram, those lying strictly below the main diagonal. Let $\alpha(\pi)$ be the partition whose parts are the lengths of the rows of $A(\pi)$, and let $\beta(\pi)$ be the partition whose parts are the lengths of the columns of $B(\pi)$. Taking $\pi$ to be the sequence $(3,2,2,2,1)$ from our earlier examples, we have $\alpha(\pi)=(3,1)$ and $\beta(\pi)=(4,2)$, as illustrated in Figure~\ref{fig: A and B tableau}.

\begin{figure}
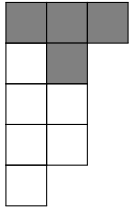

\centering
\begin{ytableau}
*(gray) & *(gray)  &*(gray) \\
&*(gray)\\
& \\
& \\
\\
\end{ytableau}
\caption{The Young diagram for $(3,2,2,2,1)$ where the boxes of $A(\pi)$ are shaded.}
\label{fig: A and B tableau}
\end{figure}
\begin{thm}[\cite{Merris03}]
Let $\pi$ be a list of nonnegative integers in nonincreasing order that has an even sum. The list $\pi$ is graphic if and only if $\alpha(\pi)$ and $\beta(\pi)$ satisfy $\sum_{i=1}^k \beta_i \geq \sum_{i=1}^k \alpha_i$ for all positive $k$ up through the shorter of the lengths of $\alpha$ and $\beta$, and $\sum_i \beta_i = \sum_i \alpha_i$.
\end{thm}

For testing part (b) of Theorem~\ref{thm: feasible tableau}, we may use a criterion due to Gale~\cite{Gale57} and Ryser~\cite{Ryser57}.

\begin{thm}[\cite{Gale57,Ryser57}]
Let $\pi=(\pi_1,\dots,\pi_p)$ and $y=(\rho_1,\dots,\rho_q)$ be lists of nonnegative integers such that the terms of $\pi$ are indexed in nonincreasing order. There is a bipartite simple graph $H$ such that $\pi$ and $\rho$ are the lists of the degrees of vertices in the respective partite sets of $H$ if and only if $\sum_i \pi_i = \sum_i \rho_i$ and for each $k$ such that $1 \leq k \leq p$, \[\sum_{i=1}^k \pi_i \leq \sum_i \min\{k,\rho_i\}.\]
\end{thm}

\section{N-switches}
\label{sec: N-switches}
A key operation on realizations of graph degree sequences is the $2$-switch (also referred to as a swap or transfer), which we now define. An \emph{alternating $4$-cycle} in a graph is a configuration consisting of four vertices $a,b,c,d$ such that $ac$ and $bd$ are edges and $ad$ and $bc$ are not edges in the graph. A \emph{2-switch} on this alternating $4$-cycle is the operation of deleting edges $ac$ and $bd$ from the graph and adding edges $ad$ and $bc$; we denote this 2-switch by $\{ac,bd\} \rightrightarrows \{ad,bc\}$. For example, in Figure~\ref{fig: 2switch} we see an alternating $4$-cycle in the graph on the left (dotted lines denote the non-adjacencies of the alternating $4$-cycle), and in the graph on the right we see the graph resulting after a $2$-switch is performed.

\begin{figure}
\centering
\includegraphics[width=6cm]{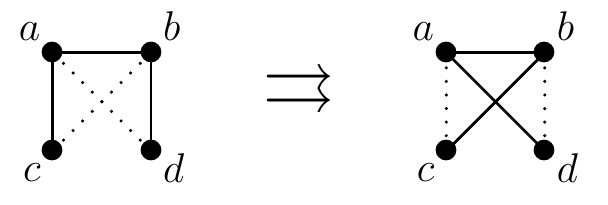}
\caption{Two realizations of $(2,2,1,1)$, along with a 2-switch between them.}
\label{fig: 2switch}
\end{figure}

Note that the two graphs illustrated in Figure~\ref{fig: 2switch} are both realizations of $(2,2,1,1)$. In general, a 2-switch preserves the degree sequence of the graph it is performed on, and in fact the following is true.

\begin{thm}[\cite{FulkersonEtAl65}; see also \cite{Petersen1891}] \label{thm: 2-switches}
In two graphs $G$ and $H$ with the same vertex set, every vertex has the same degree in both $G$ and $H$ if and only if $G$ may be transformed into $H$ via a sequence of 2-switches.
\end{thm}

Observe that an arbitrary 2-switch performed on a graph $G$ may result in a graph having a different neighborhood degree list than that of $G$. For example, if in Figure~\ref{fig: 2switch} an additional vertex $e$ was made adjacent to vertices $b$ and $d$ in both graphs, the two graphs would both have degree sequence $(3,2,2,2,1)$, but the one on the left would have NDL $((2,2,2),(3,2),(3,2),(3,1),(2))$ while the one on the right would have the NDL $((2,2,1),(3,2),(3,2),(2,2),(3))$ referred to earlier. The joint degree matrices of these two graphs would also differ (notice that the first graph has no edge between vertices with degrees 1 and 3, respectively, while the latter does). In studying graphs with a common joint degree matrix, both Stanton and Pinar~\cite{StantonPinar12} and Czabarka et al.~\cite{CDEM15} addressed what the latter paper called a \emph{restricted swap operation}, which amounts to a 2-switch $\{ac,bd\} \rightrightarrows \{ad,bc\}$ for which $a$ and $b$ have the same degree. The paper~\cite{CDEM15} completes the proof that two graphs have the same joint degree matrix if and only if they may be transformed into each other via finite sequences of restricted swap operations.

In the remainder of this section we present an analoguous switching operation for dealing with neighborhood degree lists. Define an \emph{N-switch} to be a 2-switch $\{ac,bd\} \rightrightarrows \{ad,bc\}$ such that $\deg(a)=\deg(b)$ and $\deg(c)=\deg(d)$. We now describe N-switch analogues of the previous 2-switch results.

In the following, let $\NDL(G)$ denote the neighborhood degree list of a graph $G$, and for any vertex $v$ of $G$, let $\NDL_G(v)$ (or $\NDL(v)$, if $G$ is clear from the context) denote the list of degrees that the neighbors of $v$ have in $G$ (typically ordered from largest to smallest).

First, observe that any N-switch on $G$ leaves $\NDL(v)$ unchanged for each vertex $v$ of $G$. Indeed, an N-switch is a 2-switch, so the degree of each vertex in $G$ remains the same during the operation. Furthermore, the only vertices of $G$ whose neighborhoods are changed by an N-switch are $a,b,c,d$, and for each of these vertices, one neighbor is replaced by another vertex having the same degree.

We now show that two labeled graphs have the same neighborhood degree list if and only if we can transform one into the other via a finite sequence of N-switches. We achieve this by showing that both graphs can be transformed by a sequence of N-switches to a ``canonical'' realization of the neighborhood degree list; we can then transform the first graph into the canonical realization and then reverse the other graph's sequence of N-switches, so as to transform the canonical realization into the second graph. (Note that the operation of undoing an N-switch is itself an N-switch.)

We preface our definitions and theorem with a result of Kleitman and Wang~\cite{KleitmanWang73}, stated and proved here in a slightly more general setting.

\begin{lem}[\cite{KleitmanWang73}]\label{lem: maxdegneighborhood}
Let $G$ be a graph, let $v$ be an arbitrary vertex of $G$, and let $T$ be a set of vertices in $G$ not containing $v$. Suppose that $v$ has $p$ neighbors in $T$ in the graph $G$. For any set $S$ of $p$ vertices of $T$ having the highest degrees in $G$, there exists a realization $G'$ of $\deg(G)$ in which the neighborhood of $v$, restricted to $T$, is $S$, and all neighbors of $v$ outside of $T$ are the same as they are in $G$.
\end{lem}
\begin{proof}
Given $v$, $T$, and $S$ as in the hypothesis, suppose that $v$ is not adjacent to some vertex $u$ of $S$. Then $v$ must have some neighbor $w$ that is in $T-S$. By the definition of $S$, we know that $\deg(u) \geq \deg(w)$, and since $v$ is a neighbor of $w$ that $u$ does not have, $u$ must have some neighbor $x$ that $w$ is not adjacent to. The 2-switch $\{ux, wv\} \rightrightarrows \{uv,wx\}$ leaves $v$ replaces $w$ by $u$ in the neighborhood of $v$ and leaves unchanged all other adjacency relationships involving $v$. Repeating the argument with other 2-switches, we may arrive at a graph in which the neighborhood of $v$, restricted to $T$, is $S$.
\end{proof}

\begin{dfn}
Given a graph $H$ having vertex set $\{v_1,\dots, v_n\}$, we order the vertices so that vertex $v_i$ precedes $v_j$ in the list if and only if either $\deg(v_i)>\deg(v_j)$, or $\deg(v_i)=\deg(v_j)$ and $i<j$. Given any vertex $v_k$ of $H$, the \emph{principal neighborhood of $v_k$} is the set of $\deg(v_k)$ vertices that appear first in the ordering after $v_k$ is removed. Note that the principal neighborhood of any vertex $v_k$ is uniquely determined, and by Lemma~\ref{lem: maxdegneighborhood} (with $T = \{v_1,\dots,v_n\} - \{v_k\}$), there is a realization of $\deg(H)$ in which the neighborhood of $v_k$ is its principal neighborhood.

Given a degree sequence $\pi=(\pi_1,\dots,\pi_n)$, we define the \emph{canonical realization} $R(\pi)$ of $\pi$ to be the unique graph with vertex set $V=\{v_1,\dots,v_n\}$ having the property that for each $i \in \{1,\dots,n\}$, the neighborhood of $v_i$ in $R(\pi)-\{v_1,\dots,v_{i-1}\}$ is its principal neighborhood in that graph.
\end{dfn}

Observe that we may transform any realization of $\pi$ into $R(\pi)$ by iteratively performing 2-switches as in Lemma~\ref{lem: maxdegneighborhood}, beginning with $v=v_1$ and $T=V-\{v_1\}$, then $v=v_2$ and $T=V-\{v_1,v_2\}$, and so on.

We now make the analogous definitions and observations for bipartite graphs and degree sequences.

\begin{dfn}
Given a bipartite graph $J$ having partite sets $X=\{x_1,\dots,x_p\}$ and $Y=\{y_1,\dots,y_q\}$, suppose that we additionally name the vertices of $Y$ as $\{w_1,\dots,w_q\}$ so that given any two vertices $w_i$ and $w_j$, where $w_i=y_s$ and $w_j=y_t$, the vertex $w_i$ precedes $w_j$ in the reordered list if and only if either $\deg_J(y_s)>\deg_J(y_t)$, or $\deg_J(y_s)=\deg_J(y_t)$ and $s<t$. Given any vertex $x_k$ of $X$, the \emph{principal neighborhood of $x_k$} is the set $\{w_1,\dots,w_{\deg(x_k)}\}$. As before, the principal neighborhood of any vertex in $X$ is uniquely determined, and by Lemma~\ref{lem: maxdegneighborhood} (with $T = Y$), there is a realization of $\deg(H)$ in which the neighborhood of $x_k$ is its principal neighborhood. Furthermore, the 2-switches described in the proof of Lemma~\ref{lem: maxdegneighborhood}, as performed in the bipartite graph $J$, are along alternating $4$-cycles whose edges and non-edges all contain a vertex from $X$ and a vertex from $Y$; thus the realizations obtained via these 2-switches are all bipartite with partite sets $X$ and $Y$.

Given two lists $\pi=(\pi_1,\dots,\pi_\ell)$ and $\rho=(\rho_1,\dots,\rho_n)$ of integers, where the entries are the degrees of the vertices in a bipartite graph $H$ with partite sets $X=\{x_1,\dots,x_\ell\}$ and $Y=\{y_1,\dots,y_m\}$, with $\pi_i=\deg_H(x_i)$ and $\rho_i=\deg_H(y_i)$ for all $i$, we define the \emph{canonical realization} $R(\pi; \rho)$ to be the unique bipartite graph with partite sets $X$ and $Y$ having the property that for each $i \in \{1,\dots,\ell\}$, the neighborhood of $x_i$ in $R(\pi;\rho) - \{x_1,\dots,x_{i-1}\}$ is its principal neighborhood in that graph. (Note that as vertices are deleted, the degrees of the remaining vertices may decrease, but the subscript of each vertex never changes.)
\end{dfn}

\begin{exa}
Let $\pi=(2,1,1,3,1)$ and $\rho=(1,2,3,2)$, and consider a vertex set $X=\{x_1,x_2,x_3,x_4,x_5\}$ and $Y=\{y_1,y_2,y_3,y_4\}$, where the $i$th term of $\pi$ (of $\rho$, respectively) is equal to the degree of $x_i$ (of $y_i$). The canonical realization $H=R(\pi; \rho)$ is shown in Figure~\ref{fig: R(a,b)}. Observe that the neighborhood of $x_1$ is $\{y_2,y_3\}$, since $y_3$ has the highest degree (namely, 3) in $H$ and $y_2$ is the lowest-indexed vertex of degree 2. In $H-x_1$, the neighborhood of $x_2$ is $\{y_3\}$, since $y_3$ is the lowest-indexed vertex of the highest degree in $H-x_1$. The neighborhoods of $x_3$, $x_4$, and $x_5$ similarly satisfy the requirements of the definition of $R(\pi; \rho)$.
\end{exa}
\begin{figure}
\centering
\includegraphics[width=8cm]{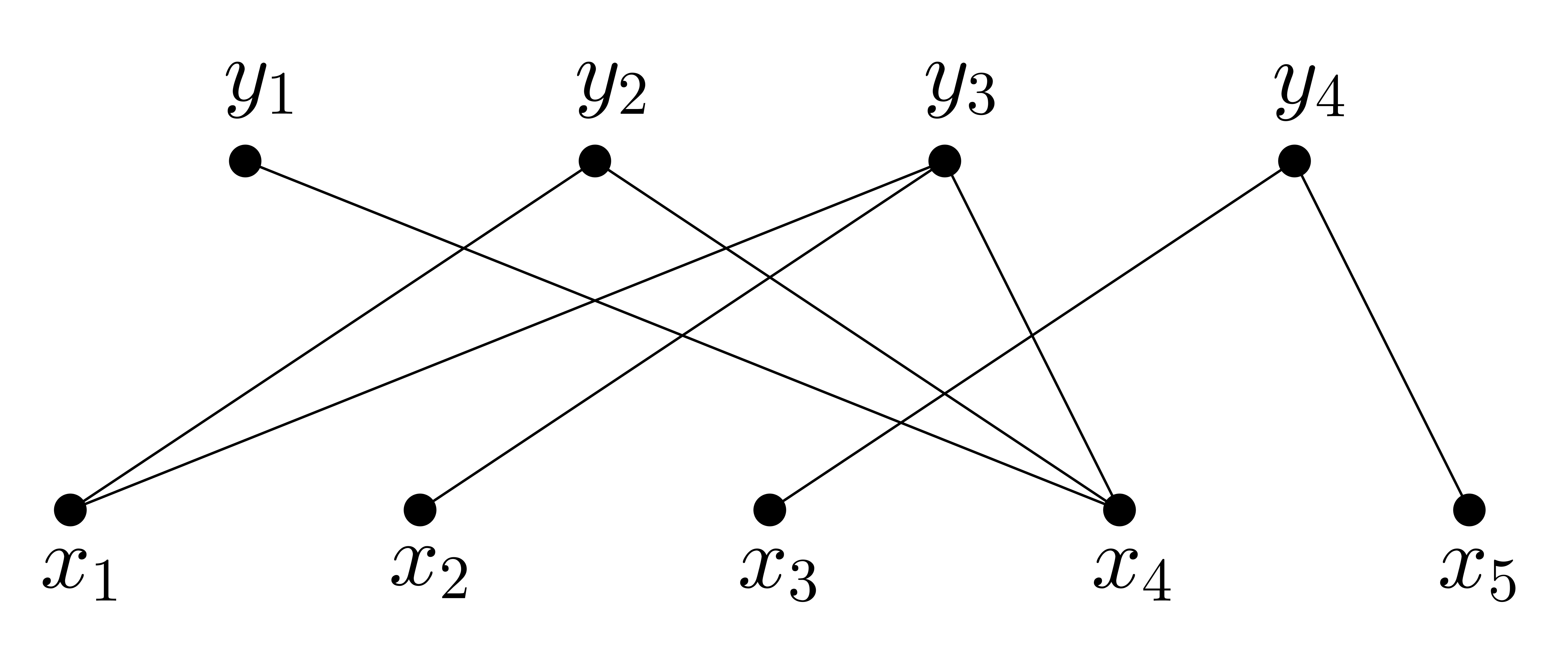}
\caption{The canonical realization $R(\pi;\rho)$ of $\pi = (2,1,1,3,1)$ and $\rho=(1,2,3,2)$.}
\label{fig: R(a,b)}
\end{figure}

Again, we may transform any bipartitioned realization of $(\pi; \rho)$ into $R(\pi;\rho)$. We accomplish this by iteratively performing 2-switches as in Lemma~\ref{lem: maxdegneighborhood}, letting $v$ equal each of $x_1,x_2,\dots$ in turn, with $T=Y$.

Having made these observations, we can now prove our main result.
\begin{thm}\label{thm: Nswitch sequence}
If $G$ and $G'$ are any two realizations of the same neighborhood degree list, then there exists a finite sequence of N-switches which, when applied to $G$, result in the graph $G'$.
\end{thm}
\begin{proof}
Let $G$ and $G'$ be two graphs with the same vertex set $V=\{v_1,v_2,\dots,v_n\}$ such that $\NDL_G(v) = \NDL_{G'}(v)$ for all $v \in V$. For each term $k$ in the degree sequence of $G$, let $V_k$ consist of all vertices of $G$ having degree $k$.

Now for each $v_i \in V$ and each integer $j$ appearing in $\NDL(v_i)$, let $\mu_i^j$ denote the number of terms in $NDL(v_i)$ equal to $j$. 

For each $k$, the induced subgraphs $G[V_k]$ and $G'[V_k]$ have the same degree sequences, since the degree of $v_i$ in each graph is precisely $\mu_i^k$. As we observed above, there exists a sequence of $2$-switches we may perform on $G[V_k]$ to transform the graph into the canonical realization $R(\deg(G[V_k]))$. There is also a similar sequence of $2$-switches that transform $G'[V_k]$ into $R(\deg(G'[V_k]))=R(\deg(G[V_k]))$; reversing this latter sequence of 2-switches and concatenating it to the earlier sequence produces a sequence of 2-switches transforming $G[V_k]$ into $G'[V_k]$. Note that the sequence of 2-switches performed on $G[V_k]$ is in fact a sequence of N-switches performed on $G$.

Now for each pair $k,\ell$ of distinct terms in the degree sequence of $G$ and $G'$, where $k < \ell$, consider the bipartite subgraphs $H_{k,\ell}$ and $H'_{k,\ell}$ of $G$ and $G'$ respectively with partite sets $V_k$ and $V_\ell$, and edge sets consisting of all edges of $G$ and $G'$, respectively, that join vertices from $V_k$ to vertices of $V_\ell$. Note that the degree of each vertex $v_i$ of $V$ appearing in $H_{k,\ell}$ is precisely $\mu_i^\ell$ (if $v_i \in V_k$) or $\mu_i^k$ (if $v_i \in V_\ell$), so each vertex has the same degree in $H_{k,\ell}$ as in $H'_{k,\ell}$. Let $\pi$ denote the list of the terms $\mu^i_\ell$, arranged in increasing order by the superscript $i$, corresponding to the vertices $v_i \in V_k$; let $\rho$ denote the list of terms $\mu_i^k$ corresponding to vertices $v_i \in V_\ell$, again arranged in increasing order by superscript.

As we observed above, there exists a sequence of $2$-switches we may perform on $H_{k,\ell}$ that places no edges within the partite sets $V_k$ and $V_\ell$ and succeeds in transforming $H_{k,\ell}$ into the canonical realization $R(\pi; \rho)$. If we append to this sequence the reverse of a sequence of partite-set-preserving 2-switches that change $H'_{k,\ell}$ into $R(\pi;\rho)$, we get a sequence of 2-switches that transform $H_{k,\ell}$ into $H'_{k,\ell}$. Note that this sequence of partite-set-preserving 2-switches performed on $H_{k,\ell}$ corresponds exactly to a sequence of N-switches performed on $G$.

Finally, note that none of the N-switches described above, in any of the subgraphs induced on $V_k$ or in the bipartite subgraph joining $V_k$ and $V_\ell$, has any effect on any of the other subgraphs involving different values of $k$ and $\ell$. Thus as long as $V$ is finite, we may perform the N-switch sequences described above on each induced subgraph $G[V_k]$ and each bipartite subgraph $H_{k,\ell}$ for all relevant values of $k$, or of $k$ and $\ell$, in turn, and obtain a sequence of N-switches that transforms $G$ into $G'$.
\end{proof}

\section{NDLs with unique labeled realizations}
 \label{sec: NDL-unique}
 
In this section we characterize the neighborhood degree lists having unique labeled realizations and the graphs that realize them.

Recall from Section~\ref{sec: graphic NDLS} the definitions of the lists $d$, $D^k$, and $D^{k,\ell}$ derived from an NDL \[T=((\tau^1_1,\dots,\tau^1_{d_1}),\dots,(\tau^n_1,\dots,\tau^n_{d_n})).\] The following theorem draws together the main ideas from the previous two sections. 

\begin{thm}
\label{thm: unique tableau char}
Let $T=((\tau^1_1,\dots,\tau^1_{d_1}),\dots,(\tau^n_1,\dots,\tau^n_{d_n}))$ be the NDL of a simple graph. The following are equivalent.

\begin{enumerate}
\item $T$ is uniquely realized by a graph;
\item The following conditions both hold for the list $d$ and the lists $D^k$, and $D^{k,\ell}$ for all $k,\ell$ appearing in $d$:
\begin{enumerate}
\item[(a)] for each distinct term $k$ in $d$, the list $D^k$ is the degree sequence of a unique graph;
\item[(b)] for each pair $k,\ell$ of distinct values appearing in $d$, the list $D^{k,\ell}$ is the degree sequence of a unique bipartitioned graph.
\end{enumerate}
\item $T$ has a realization that does not admit any N-switch.
\end{enumerate}
\end{thm}

\begin{proof}
We show that Condition 3 is equivalent to each of Conditions 1 and 2.

Since an N-switch changes the adjacencies of a graph but does not change the neighborhood degree list of that graph, it is clear that Condition 1 implies Condition 3; the converse follows from Theorem~\ref{thm: Nswitch sequence}.

Suppose that $G$ is an arbitrary realization of $T$, and for all nonnegative integers $k$, let $V_k$ denote the vertices of $G$ with degree $k$. Note that $D^k$ is the degree sequence of $G[V_k]$ for each $k$. If $D^k$ has more than one realization, then by Theorem~\ref{thm: 2-switches} there is a 2-switch possible in $G[V_k]$, and $G$ thus admits an N-switch. Similarly, note that $D^{k,\ell}$ is the degree sequence of the bipartite subgraph of $G$ with vertex set $V_k\cup V_\ell$ whose edge set contains precisely the edges of $G$ having an endpoint in each of $V_k$ and $V_\ell$; denote this graph by $G[V_k,V_\ell]$. If $D^{k,\ell}$ has more than one realization, then by Theorem~\ref{thm: Nswitch sequence} there is an N-switch possible in $G[V_k]$, from which we see that $G$ admits an N-switch. Hence Condition 3 implies Condition 2. Conversely, Condition 2 implies Condition 3, since if every realization of $T$ admits an N-switch, then we may take an N-switch $\{ac,bd\}\rightrightarrows\{ad,bc\}$ in any one of these realizations and find a corresponding 2-switch in $G[V_k,V_\ell]$, where $k=\deg(a)=\deg(b)$ and $\ell=\deg(c)=\deg(d)$ (or in $G[V_k]$, if $k=\ell$), so $D^{k,\ell}$ or $D^k$ is not the degree sequence of a unique bipartitioned graph or graph, respectively.
\end{proof}

The conditions (a) and (b) in Item 2 allow us to characterize graphs that are the unique labeled realizations of their NDLs, and the NDLs of these graphs, in terms of the lists $D^k$ and $D^{k,\ell}$. Before stating the characterization we recall a few definitions.

A \emph{threshold graph}, as defined in~\cite{ChvatalHammer73,ChvatalHammer77}, is a graph $G=(V,E)$ such that there exists a real number $t$ (the ``threshold'') and a real-valued weighting of the vertices of $G$ such that two vertices $u$ and $v$ are adjacent if and only if the sum of their weights is at least $t$. The definition of a \emph{difference graph}, as introduced in~\cite{HammerPeledSun90}, also includes a threshold and a weighting of the vertices; however, here we require that (1) no vertex receive a weight with absolute value larger than $t$, and (2) two vertices $u$ and $v$ are adjacent if and only if the absolute value of the difference of their weights is at least $t$.

The classes of threshold graphs and difference graphs both have several equivalent characterizations (see the monograph~\cite{MahadevPeled95} for a survey and detailed bibliography). We mention a few here. Let $\alpha(\pi)$ and $\beta(\pi)$ denote the partitions defined in Section~\ref{sec: graphic NDLS}.

\begin{thm}[\cite{ChvatalHammer73,FulkersonEtAl65,HammerIbarakiSimeone78, Merris03}]\label{thm: threshold char}
Let $G$ be a graph, and let $d=(d_1,\dots,d_n)$ be the degree sequence of $G$ in descending order, with $m(d)=\max\{i:d_i \geq i-1\}$. The following are equivalent, and each characterizes the class of threshold graphs.
\begin{enumerate}
\item[\textup{(1)}] The graph $G$ is the only simple graph realization of its degree sequence;
\item[\textup{(2)}] The degrees sequence $d$ satisfies the first $m(d)$ Erd\H{o}s--Gallai inequalities with equality; i.e., \[\sum_{i=1}^k d_i = k(k-1)+\sum_{i=k+1}^n \min\{k,d_i\}\] for all $k \{1,\dots,m(d)\}$;
\item[\textup{(3)}] The degree sequence $d$ satisfies $\alpha(d)=\beta(d)$.
\end{enumerate}
\end{thm}

Difference graphs have been shown to be bipartite graphs \cite{HammerPeledSun90} and appear in many ways to be the bipartite analogue of threshold graphs. In the next theorem, we list a few of their characterizations; the last condition, which resembles the last condition of Theorem~\ref{thm: threshold char}, requires a definition, which we give now.

The \emph{conjugate} of a list $\pi=(\pi_1,\dots,\pi_t)$ of nonnegative integers, denoted $\pi^*$, is defined by $\pi^*_i = \left|\{j:\pi_j \geq i\}\right|$ for $i \in \{1,\dots,\pi_1\}$. Graphically, the Young diagrams of $\pi$ and $\pi^*$ are transposes of each other.

\begin{thm}\label{thm: difference char}
Let $H$ be a bipartite graph with partite sets $X$ and $Y$, and let $\pi=(\pi_1, \pi_2, \dots, \pi_p)$ and $\rho=(\rho_1, \rho_2, \dots \rho_q)$ be the lists of degrees of the vertices in the respective partite sets, indexed in descending order in each list; also assume $\pi_1 \leq p$. The following are equivalent, and each characterizes the class of difference graphs.
\begin{enumerate}
\item[\textup{(1)}] The graph $H$ is the only realization of its two degree lists as a bipartite graph with partite sets $X$ and $Y$;
\item[\textup{(2)}] The graph $H^*$ obtained by adding to $H$ all possible edges between vertices in one of the partite sets is a threshold graph;
\item[\textup{(3)}] The degree lists $\pi$ and $\rho$ satisfy $\sum_{i} \pi_i = \sum_{i} \rho_i$ and for each $k \in \{1,\dots,p-1\}$, \[\sum_{i=1}^k \pi_i = \sum_{i=1}^q \min\{k,\rho_i\}.\]
\item[\textup{(4)}] The degree lists $\pi$ and $\rho$ satisfy $\pi^* = \rho$.
\end{enumerate}
\end{thm}
\begin{proof}
For proofs of (1)--(3), see~\cite{HammerPeledSun90,MahadevPeled95}. We show that Condition 4 is equivalent to Condition 2.

For any $H$ the addition of all possible edges in one partite set, say $X$, forces the minimum degree of $X$ to be at least as large as the maximum degree of $Y$. Denote the degree sequence of this augmented graph $H^*$ as $d$. Furthermore, we may view the Young diagram of any such graph as in Figure~\ref{subfig: diffgraphedges}: the first $p$ rows correspond to partition $X$ and consist of a $p$ by $p-1$ block for the newly added edges followed by the original edges between $X$ and $Y$, and the remaining $q$ rows for $Y$ lie below. 

Assume $H$ is a difference graph and form $H^*$ by augmenting $X$. By Theorem~\ref{thm: threshold char} we have $\alpha(d) = \beta(d)$ and for any $i$, $1 \leq i \leq p$, $p-i$ has been added to the first $p$ terms of both $\alpha(d)$ and $\beta(d)$. Thus, $\pi^* = \rho$.

Conversely, let $\pi^*=\rho$. Again, augmenting $X$ with all possible edges, we add a $p$ by $p-1$ block to form the Young diagram for $H^*$, yielding $\alpha(d)=\beta(d)$ and Condition 2.
\end{proof}

\begin{figure}[h]
\centering
\begin{subfigure}{0.45\textwidth}
\raisebox{-.47\totalheight}{\includegraphics[height=2.7cm]{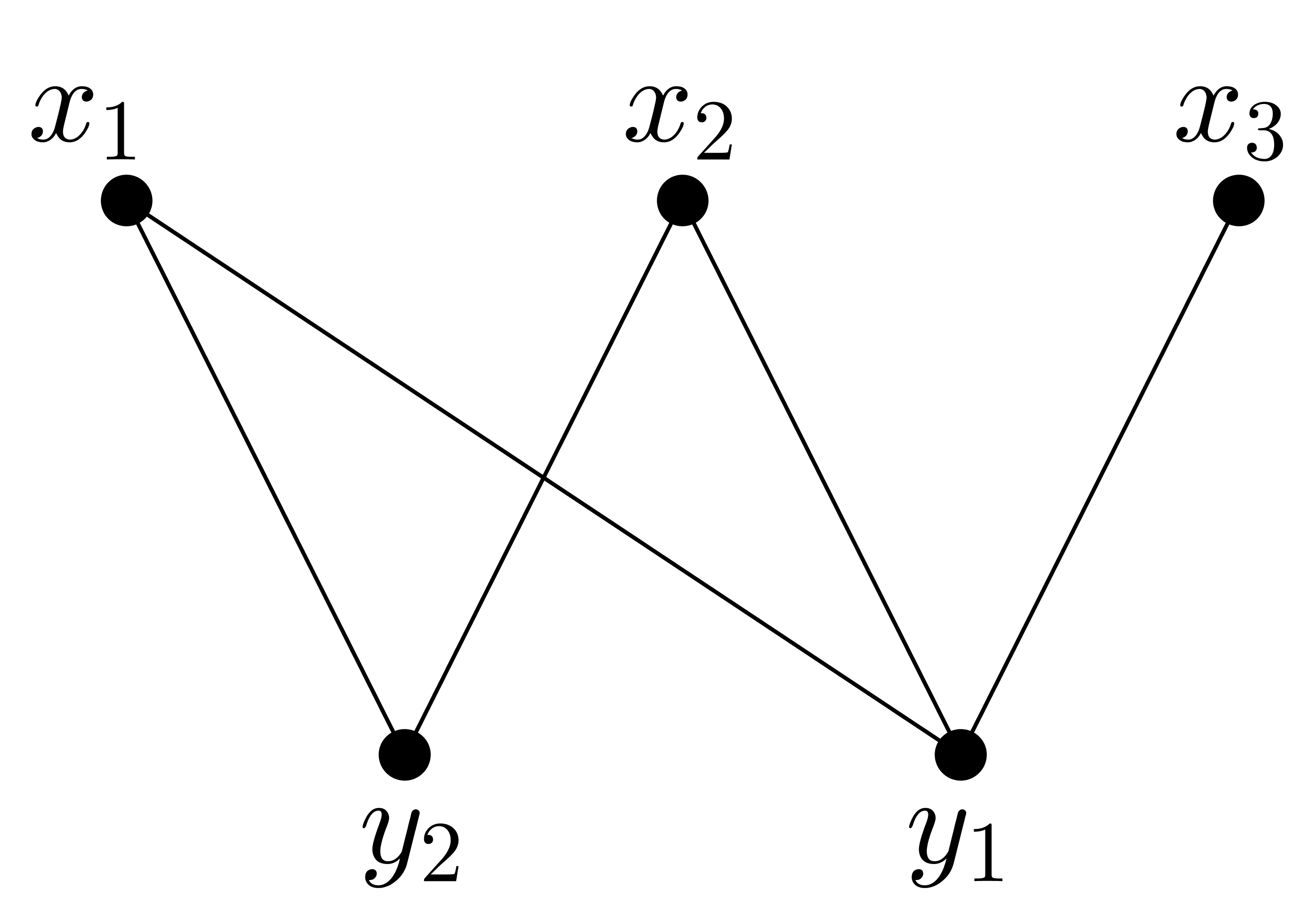}}
\hspace{.25in}
\ydiagram[*(lightgray)]
{0,0,0,0,3,2}
*[*(gray)]{2,2,1,0,3,2}
\caption{The original difference graph.}
\label{subfig: diffgraph}
\end{subfigure}
\vspace{.1in}
\begin{subfigure}{0.45\textwidth}
\raisebox{-.5\totalheight}{\includegraphics[height=2.9cm]{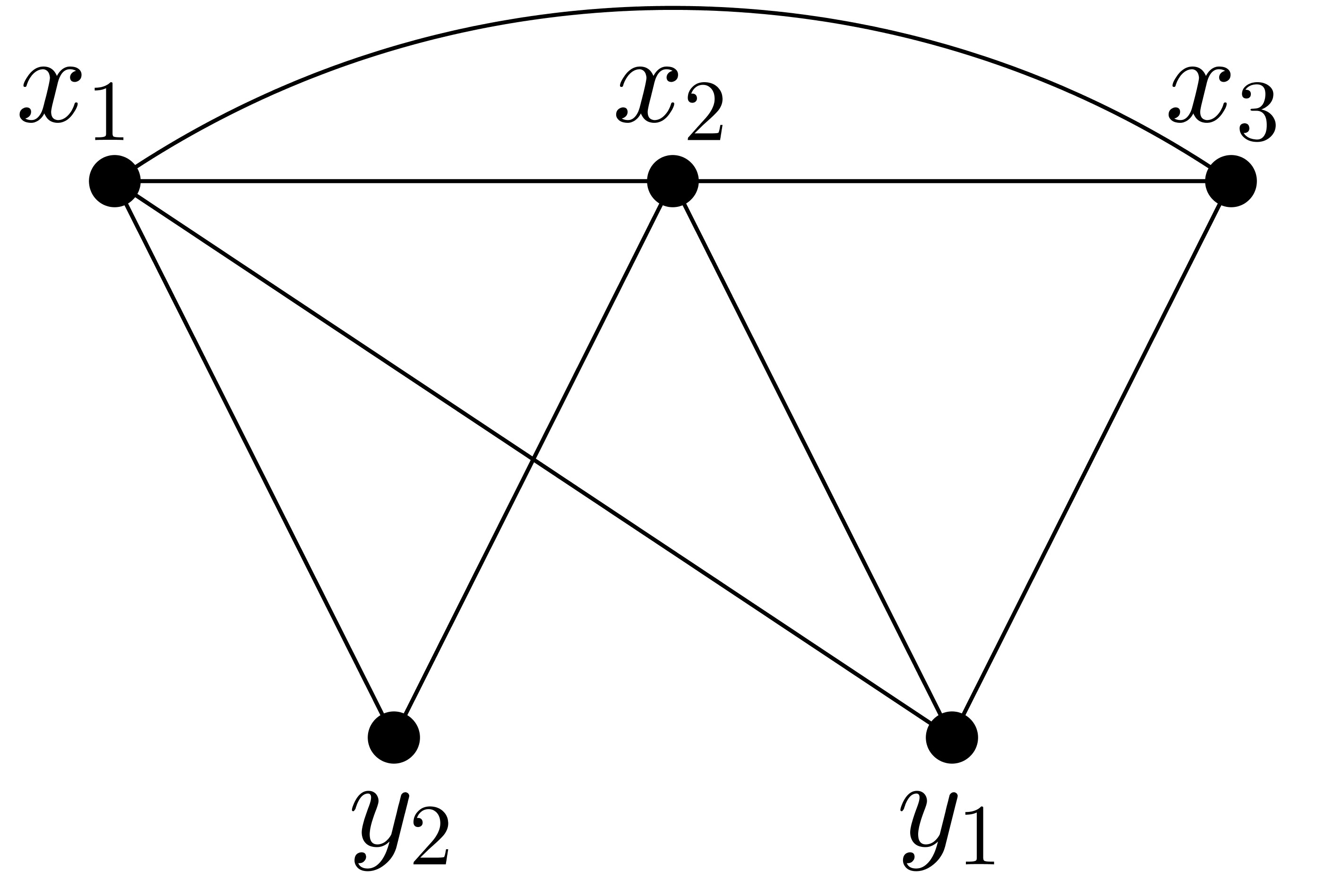}}
\hspace{.25in}
\ydiagram
[*(lightgray)]{0,0,0,3,2}
*[*(white) \bullet]{2,2,2,0,0}
*[*(gray)]{4,4,3,3,2}
\caption{The augmented graph.}
\label{subfig: diffgraphedges}
\end{subfigure}
\caption{The difference graph having degree lists $\pi=(2,2,1)$ and $\rho=(3,2)$, and the graph resulting from adding the extra edges in $X$. Also shown are the corresponding 
Young diagrams. Here, boxes corresponding to vertices in $X$ are shown in dark gray and those corresponding to vertices in $Y$ are shown in light gray.}
\label{fig: diffGraphTrans}
\end{figure}

We now continue with our characterization of NDL-unique graphs.

\begin{cor} \label{cor: subgraphs are threshold and difference}
Let $T=((\tau^1_1,\dots,\tau^1_{d_1}),\dots,(\tau^n_1,\dots,\tau^n_{d_n}))$ be the NDL of a simple graph. If $G$ is a realization of $T$, then $G$ is the unique realization of $T$ if and only if the following conditions both hold for the list $d$ and the lists $D^k$, and $D^{k,\ell}$ for all $k,\ell$ appearing in $d$:
\begin{enumerate}
\item[(a)] for each distinct term $k$ in $d$, the subgraph $G[V_k]$ is a threshold graph;
\item[(b)] for each pair $k,\ell$ of distinct values appearing in $d$, the subgraph $G[V_k,V_\ell]$ is a difference graph.
\end{enumerate}
\end{cor}
\begin{proof}
Conditions (a) and (b) follow from Theorems~\ref{thm: unique tableau char}, \ref{thm: threshold char}, and \ref{thm: difference char}.
\end{proof}

As mentioned above, the classes of both threshold and difference graphs have several remarkable properties and characterizations beyond those mentioned in the previous two theorems. In particular, each class has a forbidden subgraph characterization; threshold graphs are precisely those graphs containing no induced subgraph isomorphic to $2K_2$, $C_4$, or $P_4$, and difference graphs are precisely the bipartite graphs with no induced subgraph isomorphic to $2K_2$. Note that NDL-unique graphs are not closed under taking induced subgraphs (for instance, the graph obtained by attaching a pendant vertex to a 4-cycle is NDL-unique but contains an induced subgraph isomorphic to the non-NDL-unique $C_4$). Thus no list of forbidden induced subgraphs exists for a traditional forbidden subgraph characterization for NDL-unique graphs. However, in light of Corollary~\ref{cor: subgraphs are threshold and difference}, we can describe NDL-unique graphs in terms of subgraphs forbidden \emph{in certain positions}, namely, among vertices with the same degree or among the collection of vertices having one of two given degrees (with subgraph edges having endpoints with different degrees). Similar modifications of any of the characterizations of threshold and difference graphs may yield many different algorithms for recognizing NDL-unique graphs.

Our last characterization of NDL-unique graphs concerns their degree sequences. Recall from Section~\ref{sec: graphic NDLS} that $D^{k,\ell}$ is a bipartitioned list of lists of the form $(\mu^{\ell}_i; \mu^k_j)$ for all $i$ in $V_k$ and all $j$ in $V_{\ell}$. Let us denote this first partition as $D^{k,\ell}_X$ and the second partition as $D^{k,\ell}_Y$. We may view these partitions in the original tableau by extracting two smaller tableaux; the first is composed of those boxes from rows of length $k$ containing a $\ell$ and the second as the boxes from rows of length $\ell$ containing a $k$. 

\begin{cor}
The graph $G$ with degree sequence $d$ is the unique realization of its neighborhood degree list if and only if for every distinct $k, \ell$ in $d$, $\alpha(D^k) = \beta(D^k)$ and $D^{k,\ell}_X = (D^{k,\ell}_Y)^*$.
\end{cor}

\begin{proof}
This is a direct consequence of Theorem~\ref{thm: threshold char}, Theorem~\ref{thm: difference char}, and Corollary~\ref{cor: subgraphs are threshold and difference}.
\end{proof}


\begin{thebibliography}{99}
\bibitem{AGM08} Y.~Amanatidis, B.~Green, and M.~Mihail, Graphic realizations of joint-degree matrices. Manuscript. (2008).
\bibitem{Barrus12} M.D.~Barrus, On 2-switches and isomorphism classes, Discrete Math.~312 (2012), no.~15, 2217--2222.
\bibitem{ChvatalHammer73} V.~Chv\'{a}tal and P.L.~Hammer, Set-packing and threshold graphs, Research Report, Comp.~Sci.~Dept.~University of Waterloo, Canada CORR 73-21 (1973).
\bibitem{ChvatalHammer77} V.~Chv\'{a}tal and P.L.~Hammer, Aggregation of inequalities in integer programming. In P.L.~Hammer, E.L.~Johnson, B.H.~Korte, and G.L.~Nemhauser, editors, Studies in Integer Programming, pages 145--162. North-Holland, New York, 1977. Annals of Discrete Mathematics,  1.
\bibitem{CDEM15} \'{E}.~Czabarka, A.~Dutle, P.~Erd\H{o}s, and I.~Mikl\'{o}s, On realizations of a joint degree matrix, Discrete Appl.~Math.~181 (2015), 283--288. 
\bibitem{Das81} P.~Das, Characterization of unigraphic and unidigraphic integer-pair sequences, Discrete Math.~37 (1981), no.~1, 51--66.
\bibitem{ErdosGallai60} P.~Erd\H{o}s, and T.~Gallai, Gr\'{a}fok el\H{o}irt foksz\'{a}n\'{u} pontokkal, Matematikal Lapok~11 (1960) 264--274.
\bibitem{FulkersonEtAl65} D.R.~Fulkerson, A.J.~Hoffman, and M.H.~McAndrew, Some properties of graphs with multiple edges, Canad.~J.~Math.~17 (1965), 166--177.
\bibitem{Gale57} D.~Gale, A theorem on flows in networks, Pacific J. Math.~7 (1957), 1073--1082.
\bibitem{Hakimi62} S.L.~Hakimi, On realizability of a set of intergers as degress of the vertices of a linear graph. I, J. of SIAM~10 (1962), 496--506.
\bibitem{HammerIbarakiSimeone78} P.L.~Hammer, T.~Ibaraki, and B.~Simeone, Degree sequences of threshold graphs, Congres.~Numer.~21 (1978) 329--355.
\bibitem{HammerIbarakiSimeone81} P.L.~Hammer, T.~Ibaraki, and B.~Simeone, Threshold sequences, SIAM J.~Algebraic Discrete Methods~2 (1981) 39--49.
\bibitem{HammerPeledSun90} P.L.~Hammer, U.N.~Peled, and X.~Sun, Difference graphs, Discrete Appl.~Math.~28 (1990), 35--44.
\bibitem{Havel55} V.Havel, A remark on the existence of finite graphs, \u{C}asopis pro p\u{e}stov\'{a}n\'{i} matematiky~80 (1955), 477-480.
\bibitem{Johnson75} R.H.~Johnson, Simple separable graphs, Pacific J.~Math.~56 (1975), no.~1, 143--158.
\bibitem{Johnson80} R.H.~Johnson, Properties of unique realizations --- a survey, Discrete Math.~31 (1980), 185--192.
\bibitem{Kelly42} P.J.~Kelly, On isometric transformations, PhD Thesis, University of Wisconsin--Madison, 1942.
\bibitem{Kelly57} P.J.~Kelly, A congruence theorem for trees, Pacific J.~Math.~7 (1957), 961--968.
\bibitem{KleitmanWang73} D.J.~Kleitman, D.L.~Wang, Algorithms for constructing graphs and digraphs with given valences and factors, Discrete Math.~6 (1973) 79--88.
\bibitem{MahadevPeled95} N.V.R.~Mahadev and U.N.~Peled, Threshold graphs and related topics. Annals of Discrete Mathematics, 56. North-Holland Publishing Co., Amsterdam, 1995. 
\bibitem{Merris03} R.~Merris, Split graphs, European J. Combin.~24 (2003), 413--430.
\bibitem{PatrinosHakimi76} A.N.~Patrinos and S.L.~Hakimi, Relations between graphs and integer-pair sequences, Discrete Math.~15 (1976), no.~4, 347--358. 
\bibitem{Petersen1891} J.~Petersen, Die Theorie der regulären Graphen, Acta Math., 15 (1891), 193--220.
\bibitem{Ryser57} H.J.~Ryser, Combinatorial properties of matrices of zeros and ones, Canad.~J.~Math.~9  (1957) 371--377.
\bibitem{StantonPinar12} I.~Stanton and A.~Pinar, Constructing and sampling graphs with a prescribed joint degree distribution, ACM J.~Exp.~Algorithms, 17 (1) (2012) Article No. 3.5.
\bibitem{Tyshkevich00} R.~Tyshkevich, Decomposition of graphical sequences and unigraphs, Discrete Math.~220 (2000), no.~1--3, 201--238.
\bibitem{Ulam60} S.M.~Ulam, A collection of mathematical problems, Interscience Tracts in Pure and Applied Mathematics 8, (Interscience Publishers, 1960).
\end{thebibliography}
\end{document}